\documentclass{amsart}
\usepackage{latexsym}
\usepackage{amsmath,amsfonts,epsfig, graphics, graphicx, amsthm, amssymb}
\input{xy}
\xyoption{all}

\newtheorem{theorem}{Theorem}[section]

\newtheorem{corollary}[theorem]{Corollary}

\newtheorem{defn}{Definition}

\begin{document}

\title{Trilinear Kloosterman fractions I: partially fixed moduli and unbalanced convolutions}
\author[T. Wright]{Thomas Wright}
\begin{abstract}  
In this paper, we improve on Fouvry and Radziwiłł's results on unbalanced convolutions.  In particular, we find that if $(\alpha_m)$ and $(\beta_n)$ are sequences supported on $m\sim M$ and $n\sim M$ where $\beta_n$ is equidistributed for small moduli, then
\begin{gather*}\sum_{q\sim Q}\left|\mathop{\sum\sum}_{\substack{n\sim N,m\sim M \\ mn\equiv a\pmod q}}\alpha_m\beta_n-\frac{1}{\phi(q)}\mathop{\sum\sum}_{\substack{n\sim N,m\sim M \\ (mn,q)=1}}\alpha_m\beta_n\right|\ll \frac{X}{\log^A X},
\end{gather*}
as long as $\exp((\log x)^{\varepsilon}) \leq N \leq Q^{-11/12} X^{17/36-\varepsilon}$ with $Q\leq X^{1/2+1/66-\delta}$, along with wider bounds for $N$ if $Q\leq X^{\frac{45}{89}-\epsilon}$.  The former improves the allowable range of $N$, while the latter improves the range of $Q$.  To prove these new bounds, we improve Bettin and Chandee's famous result on trilinear forms with Kloosterman fractions in the case where the denominator has a fixed factor.
\end{abstract}
\maketitle

\section{Introduction}

A common question in number theory is this: given an arithmetic function $f$ and an integer $a\neq 0$, what is the largest $\delta>0$ for which $Q\leq X^{\frac 12+\delta}$ implies that
$$\sum_{q\sim Q}\left|\sum_{\substack{n\sim X \\ n\equiv a\pmod q}}f(n)-\frac{1}{\phi(q)}\sum_{\substack{n\sim X \\ (n,q)=1}}f(n)\right|\ll \frac{X}{\log^A X}$$
for some (or possibly for every) $A>0$?  

Often - such as in the case of the primes - such questions can be expressed in terms of convolutions.  For $m\sim M$ and $n\sim N$ with $MN=X$, define $\alpha_m$ and $\beta_n$ to be sequences of complex numbers.  We then wish to find a $\delta$ such that for $Q\leq X^{1/2+\delta}$,
\begin{gather}\label{firstquestion}\sum_{\substack{q\sim Q \\ (a,q)=1}}\left|\mathop{\sum\sum}_{\substack{n\sim N,m\sim M \\ mn\equiv a\pmod q}}\alpha_m\beta_n-\frac{1}{\phi(q)}\mathop{\sum\sum}_{\substack{n\sim N,m\sim M \\ (mn,q)=1}}\alpha_m\beta_n\right|\ll \frac{X}{\log^A X}.
\end{gather}

The cases where $\delta<0$ follow from the generalized version of the Bombieri-Vinogradov Theorem.  Cases where $\delta>0$, however, are often much more difficult.  For instance, in the case of the primes, this is the Elliot-Halberstam conjecture, and thus at present there is no fixed $\delta>0$ for which (\ref{firstquestion}) can be shown to hold if $\alpha_m=\mu(m)$ and $\beta_n=\log n$. 

In order for it to be possible for (\ref{firstquestion}) to hold, it is generally required that at least one of the two functions is equidistributed for small moduli.  We call such equidistribution the \textit{Siegel-Walfisz condition}, defined as follows:
\begin{defn}
An arithmetic function $\beta_n$ is defined to be \textit{Siegel-Walfisz} if there exists
an integer $k>0$ such that for any fixed $A>0$, uniformly in $x\geq 2$, $q>|a|\geq 1$, $r\geq 1$, and $(a,q)=1$, we have
$$\sum_{\substack{x\leq n\leq 2x \\ n\equiv a\pmod q \\ (n,r)=1}}\beta_n=\frac{1}{\phi(q)}\sum_{\substack{x\leq n\leq 2x \\ (n,qr)=1}}\beta_n+O\left(x(\log x)^{-A}\tau_k(r)\right).$$
\end{defn}
Here, $\tau_k$ denotes the $k$-fold divisor function.

In 2018, Fouvry and Radziwill showed that (\ref{firstquestion}) holds for small positive values of $\delta>0$ as long as $N$ is fairly small relative to $M$.  In particular, Corollary 1.1 of \cite{FR} was the following:
\begin{corollary}\label{Cor1.1}
Let $k>0$ and $\varepsilon>0$ be given. Let 
$\alpha_m$ and $\beta_n$ be two sequences of complex numbers such that
$$|\alpha_m|\leq \tau_k(m)
\quad\text{and}\quad |\beta_n|\leq \tau_k(n)$$
for every $m,n>1$.  Suppose that $\beta_n$ satisfies the Siegel-Walfisz condition. 
Then for every $A>0$, uniformly for $M,N\geq 2$ with 
$$\frac{MN}{2}\leq X\leq 4MN,$$
we have
\begin{gather*}
\sum_{\substack{Q\leq q\leq 2Q \\ (q,a)=1}}
\left|\sum_{\substack{X<mn\leq 2X \\ mn\equiv a \pmod q}}
\alpha_m \beta_n-\frac{1}{\varphi(q)}\sum_{\substack{X<mn\leq 2X \\ (mn,q)=1}}
\alpha_m \beta_n\right|\ll X (\log X)^{-A},
\end{gather*}
provided that one of the following three conditions holds:
\begin{enumerate}
\item[(i)] $\exp((\log X)^{\varepsilon}) \leq N \leq Q^{-11/12} X^{17/36-\varepsilon}$
and $1\leq |a|\leq X/12$;
\item[(ii)] $\exp((\log X)^{\varepsilon}) \leq N \leq X^{7/90-\varepsilon}$,
$Q\leq X^{53/105-\varepsilon}$,
and $1\leq |a|\leq X/12$;
\item[(iii)] $\exp((\log X)^{\varepsilon}) \leq N \leq X^{101/630-\varepsilon}$,
$Q\leq X^{53/105-\varepsilon}$,
and $1\leq |a|\leq (X/4)^{\varepsilon/1000}$.
\end{enumerate}
\end{corollary}
A key step in this result involves a bound for trilinear forms with Kloosterman fractions discovered by Bettin and Chandee \cite{BC}.  Define
$$\mathcal B(M,N,A)=\mathop{\sum\sum\sum}_{\substack{a\sim A, m\sim M,n\sim N \\ (m,n)=1}}\alpha_m\beta_n\nu_a e\left(\vartheta \frac{a\overline{m}}{n}\right),$$
where $\vartheta$ is a nonzero integer.

In \cite{BC}, the authors then prove the following:
\begin{theorem}\label{BCThm}
\begin{align*}
\mathcal B(M,N,A)\ll &||\alpha||||\beta||||\nu||\left(1+\frac{|\vartheta|A}{MN}\right)^\frac 12\\
&\times\left((AMN)^{\frac 7{20}+\varepsilon}(M+N)^\frac 14+(AMN)^{\frac 38+\varepsilon}(AN+AM)^\frac 18\right).
\end{align*}

\end{theorem}

\section{New Ideas}

In this paper, we alter the Bettin-Chandee result to make it slightly more helpful for the Fouvry-Radziwill result.  

We recall that the \cite{FR} work considers functions $g(m)$ and $h(n)$ with $m\sim M$ and $n\sim N$ where $MN$ is roughly of size $X$ but $\exp((\log X)^\theta)\leq N\leq X^{1/72-\epsilon}$ (i.e. $N$ is much smaller than $M$).  When the authors of \cite{FR} then apply the result of \cite{BC}, they have a sum of the form

\begin{gather}\label{FRlike}\mathop{\sum\sum\sum}_{\substack{a\sim A, b\sim Q,q\sim Q \\ (b,Rq)=1}}f_1(b)f_2(q)f_3(a) e\left(\vartheta \frac{a\overline{b}}{Rq}\right)\end{gather}
for a parameter $R\leq 2N$.  Translating their work into the language of Bettin and Chandee, they then define their function $\alpha_{n}$ to be such that $n$ lies on intervals of the form $[RQ,2RQ]$ with the caveat that $\alpha_n=0$ if $R\nmid n$.  This allows them to apply Theorem \ref{BCThm}, but at the slight cost of having to replace $Q$ with $NQ$ in the bound.

However, the original proof of Bettin and Chandee (and the preceding work of Duke, Friedlander, and Iwaniec) already has a useful mechanism that one could exploit to find a better bound for something like (\ref{FRlike}).  In \cite{BC}, the authors define
\begin{gather}\label{C1}\mathcal C_1(M,N,A)=\sum_{m\sim M}\left|\mathop{\sum\sum}_{\substack{a\sim A, n\sim N \\ (m,n)=1}}\beta_n\nu_a e\left(\vartheta \frac{a\overline{m}}{n}\right)\right|^2\end{gather}
and then introduce the so-called "complementary divisor" $b$, considering
$$\mathcal C_{b}(M,N,A;\beta,\nu)=\sum_{\substack{m\in \mathcal M \\ (m,b)=1}}\left|\mathop{\sum\sum}_{\substack{a\in \mathcal A, n\in \mathcal N \\ (m,n)=1}}\beta_{n}\nu_a e\left(\theta \frac{a\overline{m}}{nb}\right)\right|^2.$$
The authors then relate $\mathcal C_1$ and $\mathcal C_b$ back to $\mathcal B$, bounding the latter via estimates on $\mathcal C_1$ and $\mathcal C_b$ as well as Cauchy-Schwarz.

If one were to consider instead
$$\sum_{m\sim M}\left|\mathop{\sum\sum}_{\substack{a\sim A, n\sim N \\ (m,n)=1}}\beta_n\nu_a e\left(\vartheta \frac{a\overline{m}}{Rn}\right)\right|^2$$
and 
$$\sum_{m\sim M}\left|\mathop{\sum\sum}_{\substack{a\sim A, n\sim N \\ (m,n)=1}}\beta_n\nu_a e\left(\vartheta \frac{a\overline{m}}{bRn}\right)\right|^2,$$
the sum now already has a built-in complementary divisor in $R$, and hence one could use the same process the authors used to bound $\mathcal C_b$ to find bounds for these expressions as well.  Applying this process and Cauchy-Schwarz, we can then give a bound for
\begin{gather*}\mathcal B(M,N,A;R):=\mathop{\sum\sum\sum}_{\substack{a\in \mathcal A, m\in \mathcal M,n\in \mathcal N \\ (m,nR)=1}}\alpha_m\beta_n\nu_a e\left(\theta \frac{a\overline{m}}{nR}\right)
\end{gather*}
in the same way as before.  This is the mechanism in the Bettin-Chandee proof that we use to find improvements on the Fouvry-Radziwill bound.

In particular, we will prove the following version of Bettin and Chandee's result:
\begin{theorem}\label{BCimprovement}
Let $M\ll N^2$ and $R\ll M^A$ for some large $A$.  Then
\begin{align*}\mathcal B(M,N,A;R)\ll &M^\epsilon ||\alpha||||\nu||||\beta||(AMN)^\frac 12R^\frac 14\left(1+\frac{|\vartheta|A}{MN}\right)^\frac 14\\
&\times \left(\frac{1}{N^{\frac 18}}+\frac{R^\frac 18N^\frac 18}{M^{\frac 14}}+\frac{M^\frac 1{10}}{R^{\frac{3}{20}}A^{\frac{1}{20}}N^{\frac{3}{20}}}+\frac{N^{\frac{3}{20}}}{A^{\frac{3}{20}}M^{\frac 15}}+\frac{N^\frac 38}{M^\frac 12}\right)
\end{align*}
\end{theorem}
When $R=1$, this is (7.2) of \cite{BC}.  However, it is clear that when $R$ is much larger than 1, this theorem will yield a smaller bound than would simply substituting $NR$ for $N$ in Theorem \ref{BCThm}.

Armed with this theorem, we record the following improvement of Corollary \ref{Cor1.1}:
\begin{corollary}\label{cor1.1.5}
Let $k>0$ and $\varepsilon>0$ be given. Let 
$\alpha_m$ and $\beta_n$ be $k$-fold divisor-bounded sequences of complex numbers as in Corollary \ref{Cor1.1} with $\beta_n$ satisfying the Siegel-Walfisz condition.  Then for every $A>0$, uniformly for $M,N\geq 2$ with 
$$\frac{MN}{2}\leq X\leq 4MN,$$
we have
\begin{equation}
\sum_{\substack{Q\leq q\leq 2Q \\ (q,a)=1}}\left|\mathop{\sum\sum}_{\substack{X<mn\leq 2X \\ mn\equiv a \pmod q}}
\alpha_m \beta_n-\frac{1}{\varphi(q)}\mathop{\sum\sum}_{\substack{X<mn\leq 2X \\ (mn,q)=1}}\alpha_m \beta_n\right|\ll_A x (\log X)^{-A},
\end{equation}
provided that one of the following three conditions holds:
\begin{enumerate}
\item[(i)] $\exp((\log X)^{\varepsilon}) \leq N \leq Q^{-33/28} X^{17/28-\varepsilon}$
and $1\leq |a|\leq X/12$;
\item[(ii)] $\exp((\log X)^{\varepsilon}) \leq N \leq X^{7/90-\varepsilon}$,
$Q\leq X^{45/89-\varepsilon}$,
and $1\leq |a|\leq X/12$;
\item[(iii)] $\exp((\log X)^{\varepsilon}) \leq N \leq X^{101/630-\varepsilon}$,
$Q\leq X^{45/89-\varepsilon}$,
and $1\leq |a|\leq (X/4)^{\varepsilon/1000}$.
\end{enumerate}
\end{corollary}
Note that in the extremal case where $N$ is very small, the bound in (i) gives $Q\leq X^{1/2+1/66-\epsilon}$, which is the same as in \cite{FR}.  This is of course because the savings are entirely dependent on the size of $N$, and thus if $N$ is very small we have no appreciable savings.  On the other hand, if $Q=X^{\frac 12+\epsilon}$ then (i) of our corollary gives $N\leq X^{1/56-\epsilon'}$ for some $\epsilon'$ depending on $\epsilon$, which is an improvement on the $N\leq X^{1/72-\epsilon'}$ from Corollary \ref{Cor1.1}.

In their paper, Fouvry and Radziwill state Corollary \ref{Cor1.1} as the consequence of an estimate involving Linnik's dispersion.  Define
$$E(\alpha, \beta, M, N, q, a):=\mathop{\sum\sum}_{\substack{m \sim M \\ n \sim N \\ mn \equiv a \,(\mathrm{mod}\, q)}}
\alpha_m \beta_n-\frac{1}{\varphi(q)}\mathop{\sum\sum}_{\substack{m \sim M \\ n \sim N \\ (mn,q)=1}}
\alpha_m \beta_n,$$
and define
$$\Delta(\alpha, \beta, M, N, Q, a):=\sum_{\substack{q \sim Q \\ (q,a)=1}}
\left| E(\alpha, \beta, M, N, q, a) \right|.$$
For the contribution of the small moduli to $\Delta(\alpha, \beta, M, N, Q, a)$, we have
$$\mathcal E^*(\beta, N, Q):=\sum_{\delta}\sum_{v \sim Q/\delta}
\sum_{\substack{ \delta' \,(\mathrm{mod}\,\delta) \\ (\delta',\delta)=1}}
\left| E(\beta, N, \delta, \delta'; v) \right|^2.$$
Note that if $\beta$ is Siegel-Walfisz and satisfies
$|\beta_n| \leq \tau_k(n)$ for some $k > 0$, and if $N > Q^\varepsilon$, then
$$\mathcal E^*(\beta, N, Q)\ll_A N^2 Q (\log N)^{-A}.$$
In this paper, we will prove the following:
\begin{theorem}\label{dispthm} Let $k$, $\varepsilon$, $\alpha_m$, and $\beta_n$ be as in the previous corollary.
Let $k\geq 1$ be an integer and let $\varepsilon>0$ be given. 
$$M>Q (MN)^{\varepsilon}, \qquad M>N>D^{10}.$$
Then, for all integers $1\leq |a|\leq X/3$, we have
\begin{align*}&|\Delta(\alpha,\beta,M,N,Q,a)|\\
&\ll ||\alpha||\left(MQ^{-1}\mathcal E^*(\beta,N,Q)+(\log X)^\kappa N^2Q+(\log X)^\kappa N^2D^{-\frac 12}M+D^CX^\epsilon\left(Q^{\frac{15}{8}}N^{\frac{11}{4}}+M^\frac{3}{20}Q^{\frac{33}{20}}N^{\frac{51}{20}}\right)
\right)^\frac 12.
\end{align*}
\end{theorem}
In \cite{FR}, the last term is
$$D^CX^\epsilon\left(Q^{\frac{15}{8}}N^{\frac{23}{8}}+M^\frac{3}{20}Q^{\frac{33}{20}}N^{\frac{59}{20}}\right).$$
Hence we have a savings of $N^\frac 18$ in the first term and $N^{\frac 25}$ in the second.

We note that Theorem \ref{dispthm} also gives an improvement of \cite[Corollary 1.5]{FR}.

\begin{corollary}
Let $\varepsilon>0$ and $k>0$ be given. Let $\lambda_d$ 
be a sequence of complex numbers with $|\lambda_d|\leq \tau_k(d)$. Then
$$\sum_{\substack{q \sim Q \\ (q,a)=1}}\left|\sum_{\substack{X < n\leq 2X \\ n \equiv a \pmod q}}
\sum_{\substack{d \mid n \\ d\leq z}} \lambda_d-\frac{1}{\varphi(q)}\sum_{\substack{X < n\leq 2X \\ (n,q)=1}}
\sum_{\substack{d \mid n \\ d\leq z}} \lambda_d\right|\ll_A \frac{X}{(\log X)^A},$$
provided that one of the following three conditions holds:
\begin{enumerate}
\item[(i)] $X\geq 12$, $z\leq X^{45/89-\varepsilon}$, 
$X^{1-\varepsilon}>Q>X^{529/630+\varepsilon}$, 
and $1\leq |a|\leq X^{\varepsilon/10000}$;
\item[(ii)] $X\geq 12$, $z\leq X^{45/89-\varepsilon}$, 
$X^{1-\varepsilon}>Q>X^{83/90+\varepsilon}$, 
and $1\leq |a|\leq X^{1-3\varepsilon}$;
\item[(iii)] $X\geq 12$, $z\leq X^{1/2+\delta-\varepsilon}$, 
$X^{1-\varepsilon}>Q>X^{(55+66\delta)/56+\varepsilon}$, 
and $1\leq |a|\leq X^{1-3\varepsilon}$ for any fixed 
$0 < \delta < \tfrac{1}{66}$.
In this case the implicit constant in $\ll_A$ depends additionally on $\delta$.
\end{enumerate}
\end{corollary}
The proof of this corollary is essentially identical to that of Corollary 1.5 in \cite{FR} with the changes made in our proof of Corollary \ref{cor1.1.5} above.  Hence, we omit this proof.

\section{Proof of Theorem \ref{BCimprovement}: trilinear forms with Kloosterman fractions and a fixed factor in the denominator}

First, we give the proof for our improved version of Bettin and Chandee's result.

Let $\mathcal N_{x}$ denote the interval $[N/x,2N/x]$.  Recalling the definitions of $\mathcal C_1$ and $C_b$, we define $C_{b;R}$ as 
\begin{align*}\mathcal C_{b;R}&(M,N,A;\beta,\nu)=\sum_{\substack{m\in \mathcal M \\ (m,R)=1}}\left|\sum_{\substack{r|R \\ r\mbox{ }square-free}}\mathop{\sum\sum}_{\substack{a\in \mathcal A, n\in \mathcal N_r \\ (mR,n)=1}}\beta_{nr}\nu_a e\left(\theta \frac{a\overline{m}}{nbrR}\right)\right|^2\\
\end{align*}
Note that 
\begin{gather}\label{CS1}\mathcal B(M,N,A;R)\ll ||\alpha||\mathcal C_{1;R}(M,N,A;R;\beta,\nu)^{\frac 12}.\end{gather}
In (5.2), the authors of \cite{BC} find that for any $b$,
\begin{align*}
\mathcal C_{b}(M,N,A;\beta,\nu)\ll & ||\beta||^2||\nu||^2M^\epsilon \left(1+\frac{|\vartheta|A}{bMN}\right)^\frac 12\bigg{(}AM(bN)^\frac 12+b^\frac 34AM^\frac 12N^\frac 54+\frac{AM^\frac 65N^\frac{1}{10}}{b^\frac 25}\\
&+b^\frac 15A^\frac 25M^\frac 65N^\frac{7}{10}+b^\frac 12A^\frac{7}{10}M^\frac 35N^{\frac{13}{10}}+b^\frac 12AN^\frac 74\bigg{)}.
\end{align*}
By our notation, we then have
\begin{align*}\mathcal C_{1;R}&(M,N,A;\beta,\nu)\\
=&\sum_{\substack{m\in \mathcal M \\ (m,R)=1}}\left|\sum_{\substack{r|R \\ r\mbox{ }square-free}}\mathop{\sum\sum}_{\substack{a\in \mathcal A, n\in \mathcal N_r \\ (mR,n)=1}}\beta_{nr}\nu_a e\left(\theta \frac{a\overline{m}}{nrR}\right)\right|^2\\
=&\sum_{\substack{m\in \mathcal M \\ (m,R)=1}}\left|\sum_{\substack{b\mbox{ }square-full \\ (m,b)=1}}\sum_{\substack{r|R \\ r\mbox{ }square-free}}\mathop{\sum\sum}_{\substack{a\in \mathcal A, n\in \mathcal N_{br} \\ (mbR,n)=1\\n\mbox{ }square-free}}\beta_{nbr}\nu_a e\left(\theta \frac{a\overline{m}}{nbrR}\right)\right|^2\\
\ll &\tau(R)\sum_{\substack{b'\mbox{ }square-full \\ (m,b')=1}}(b')^{-\frac 12}\sum_{\substack{b\mbox{ }square-full \\ (m,b)=1}}b^{\frac 12}\sum_{\substack{r|R \\ r\mbox{ }square-free}}\sum_{\substack{m\in \mathcal M \\ (m,R)=1}}\left|\mathop{\sum\sum}_{\substack{a\in \mathcal A, n\in \mathcal N_{br} \\ (mbR,n)=1\\n\mbox{ }square-free}}\beta_{nbr}\nu_a e\left(\theta \frac{a\overline{m}}{nbrR}\right)\right|^2\\
\ll &\tau(R)\sum_{\substack{b\mbox{ }square-full \\ (m,b)=1}}b^{\frac 12}\sum_{\substack{r|R \\ r\mbox{ }square-free}}\mathcal C_{brR}(M,N/br,A;\beta_{br},\nu),
\end{align*}
where $\beta_{br}(n)=\mu(n)^2\beta_{brn}$.  For the interval where $b>B$, we use the trivial bound that 
$$\mathcal C_{brR}(M,N/br,A;\beta_{br},\nu)\ll ||\nu||^2\frac{AMN}{br}||\beta_{br}||^2,$$
and hence these terms contribute
$$\ll ||\beta||^2||\nu||^2\frac{ANM^{1+\epsilon}}{B^\frac 12}.$$
Taking $B=N^\frac 12$, we then have that the above is
$$\ll ||\beta||^2||\nu||^2AN^\frac 34M^{1+\epsilon}.$$
For the terms with $b\leq B$, we apply our bound for $\mathcal C_{b}$ above, finding
\begin{align*}\tau(R)&\sum_{\substack{b\mbox{ }square-full \\ (m,b)=1 \\ b\leq B}}b^{\frac 12}\sum_{\substack{r|R \\ r\mbox{ }square-free }}\mathcal C_{brR}(M,N/br,A;\beta_{br},\nu)\\
\ll &M^\epsilon ||\nu||^2\sum_{\substack{b\mbox{ }square-full \\ (m,b)=1 \\ b\leq B}}\sum_{\substack{r|R \\ r\mbox{ }square-free}}||\beta_{br}||^2\left(1+\frac{|\vartheta|A}{MN}\right)^\frac 12\bigg{(}\frac{AMR^{\frac 12}N^\frac 12}{r^\frac 12}+\frac{R^\frac 34AM^\frac 12N^\frac 54}{b^\frac 12r^\frac 12}+\frac{AM^\frac 65N^\frac{1}{10}}{R^\frac 25b^\frac 12r^\frac 12}\\
&\phantom{abcde}+\frac{R^\frac 15A^\frac 25M^\frac 65N^\frac{7}{10}}{b^\frac 12r^\frac 12}+\frac{R^\frac 12A^\frac{7}{10}M^\frac 35N^{\frac{13}{10}}}{b^\frac 45r^\frac 45}+\frac{R^\frac 12AN^\frac 74}{b^\frac 54r^\frac 54}\bigg{)}\\
\ll &M^\epsilon ||\nu||^2||\beta||^2\left(1+\frac{|\vartheta|A}{MN}\right)^\frac 12\bigg{(}AMR^{\frac 12}N^\frac 12B^{\frac 12}+R^\frac 34AM^\frac 12N^\frac 54+\frac{AM^\frac 65N^\frac{1}{10}}{R^\frac 25}\\
&\phantom{abcde}+R^\frac 15A^\frac 25M^\frac 65N^\frac{7}{10}+R^\frac 12A^\frac{7}{10}M^\frac 35N^{\frac{13}{10}}+R^\frac 12AN^\frac 74\bigg{)}.
\end{align*}
Letting $B=N^\frac 12$ as above, we use the fact that $M\ll N^2$, which allows us to write that $\frac{AM^\frac 65N^\frac{1}{10}}{R^\frac 25}\leq AMR^{\frac 12}N^\frac 34.$  So 
\begin{align*}&\mathcal C_{1;R}(M,N,A;\beta,\nu)\\
&\ll M^\epsilon ||\nu||^2||\beta||^2\left(1+\frac{|\vartheta|A}{MN}\right)^\frac 12\bigg{(}AMR^{\frac 12}N^\frac 34+R^\frac 34AM^\frac 12N^\frac 54+R^\frac 15A^\frac 25M^\frac 65N^\frac{7}{10}+R^\frac 12A^\frac{7}{10}M^\frac 35N^{\frac{13}{10}}+R^\frac 12AN^\frac 74\bigg{)}\\
&=M^\epsilon ||\nu||^2||\beta||^2AMNR^\frac 12\left(1+\frac{|\vartheta|A}{MN}\right)^\frac 12\left(\frac{1}{N^{\frac 14}}+\frac{R^\frac 14N^\frac 14}{M^{\frac 12}}+\frac{M^\frac 15}{R^{\frac{3}{10}}A^{\frac{3}{5}}N^{\frac{3}{10}}}+\frac{N^{\frac{3}{10}}}{A^{\frac{3}{10}}M^{\frac 25}}+\frac{N^\frac 34}{M}\right).
\end{align*}
Thus by (\ref{CS1}),
\begin{align*}&\mathcal B(M,N,A;R)\\
&\ll ||\alpha||||\nu||||\beta||
\left(1+\frac{|\vartheta|A}{MN}\right)^\frac 14\\
&\phantom{ti}\times \bigg{(}A^\frac 12M^\frac 12R^{\frac 14}N^\frac 38+R^\frac 38A^\frac 12M^\frac 14N^\frac 58+R^\frac 1{10}A^\frac 2{10}M^\frac 35N^\frac{7}{20}+R^\frac 14A^\frac{7}{20}M^\frac 3{10}N^{\frac{13}{20}}+R^\frac 14A^\frac 12N^\frac 78\bigg{)}\\
&\ll M^\epsilon ||\alpha||||\nu||||\beta||(AMN)^\frac 12R^\frac 14\left(1+\frac{|\vartheta|A}{MN}\right)^\frac 14\left(\frac{1}{N^{\frac 18}}+\frac{R^\frac 18N^\frac 18}{M^{\frac 14}}+\frac{M^\frac 1{10}}{R^{\frac{3}{20}}A^{\frac{3}{10}}N^{\frac{3}{20}}}+\frac{N^{\frac{3}{20}}}{A^{\frac{3}{20}}M^{\frac 15}}+\frac{N^\frac 38}{M^\frac 12}\right).
\end{align*}


\section{Proof of Theorem \ref{dispthm}: an improved bound for dispersion}
\begin{proof}
Now we apply Theorem \ref{BCimprovement} to the work of Fouvry and Radziwill to prove Theorem \ref{dispthm}.

Recall that 
$$\Delta(\alpha, \beta, M, N, Q, a)=\sum_{\substack{q\sim Q \\ (a,q)=1}}\left|\mathop{\sum\sum}_{\substack{m \sim M \\ n \sim N \\ mn \equiv a \,(\mathrm{mod}\, q)}}\alpha_m \beta_n-\frac{1}{\varphi(q)}\mathop{\sum\sum}_{\substack{m \sim M \\ n \sim N \\ (mn,q)=1}}
\alpha_m \beta_n\right|.$$
Define $c_q=\pm 1$ or 0 to be such that if $(a,q)=1$ then $c_q=\pm 1$ with the sign matching the sign of the term inside of the absolute value above, while if $(a,q)>1$ then $c_q=0$.  In \cite{FR}, the authors bound this quantity by taking a well-chosen smooth, compactly supported function $\psi$ with bounded derivatives and using the fact that
\begin{gather}\label{Fourier1}\sum_{m\equiv a\pmod q}\psi\left(\frac mM\right)=\hat{\psi}(0)\frac Mq+\frac Mq\sum_{0<|h|\leq H}e\left(\frac{ah}{q}\right)\hat{\psi}\left(\frac{h}{q/M}\right)+O\left(M^{-1}\right),
\end{gather}
and
\begin{gather}\label{Fourier2}\sum_{(m,q)=1}\psi\left(\frac mM\right)=\frac{\varphi(q)}{q}\hat{\psi}(0)M+O\left(\tau(q)(\log 2M)^2\right).
\end{gather}
Applying Cauchy-Schwarz, the authors bound $\Delta$ as
$$|\Delta(\alpha,\beta,M,N,Q,a)|\ll ||\alpha||\left(W(Q)-2\Re V(Q)+U(Q)\right)^\frac 12,$$
where
\begin{gather*}U(Q)=\sum_m\psi\left(\frac{m}{M}\right)\left|\mathop{\sum\sum}_{\substack{q\sim Q, n\sim N \\ (mn,q)=1}}\frac{c_q}{\varphi(q)}\beta_n\right|^2,\\
V(Q)=\sum_m\psi\left(\frac{m}{M}\right)\left(\mathop{\sum\sum}_{\substack{q\sim Q, n\sim N \\ mn\equiv a \pmod q}}c_q \beta_n-
\mathop{\sum\sum}_{\substack{q\sim Q, n\sim N \\ (mn,q)=1}}\frac{c_q}{\varphi(q)} \beta_n\right)^2,\\
W(Q)=\sum_m\psi\left(\frac{m}{M}\right)\left|\mathop{\sum\sum}_{\substack{q\sim Q, n\sim N \\ mn\equiv a \pmod q}}c_q \beta_n
\right|^2 .
\end{gather*}
We use the same estimates for $U(Q)$ and $V(Q)$ as were found in the original paper.  

For $W(Q)$, Fouvry and Radziwill apply (\ref{Fourier1}) in \cite[(26) and (27)]{FR} to find that 
$$W(Q)=\widetilde{W}^{MT}(Q)+\widetilde{W}^{Err1}(Q)+\widetilde{W}^{Err2}(Q)+O\left(MN^2Q^{-1}X^{\eta-\frac \varepsilon 2}+X^{1+\eta+\frac\varepsilon 2}\right),$$
with each of the three terms corresponding to the three terms in (\ref{Fourier1}) and where $\eta>0$ is arbitrarily small.  We will use the same estimates for $\widetilde{W}^{MT}(Q)$ and $\widetilde{W}^{Err2}(Q)$ as in the original paper and will focus our efforts on improving $\widetilde{W}^{Err1}(Q)$.

Take $D<N^{10}$ as described in the statement of Theorem \ref{dispthm}.  In (48) of \cite{FR}, the authors arrive at the following estimate:
\begin{align*}\widetilde{W}^{Err1}\ll D C^{4} L^{100}MQ^{-2} \sup_{\substack{(\alpha_1,\ldots,\alpha_5)\\(d,d_1,\delta,\delta_1,\delta_2)}}
\sum_{\nu'_1\leq 2N}\sum_{\nu_2\leq 2N}\bigg{|}&\sum_{1\leq h\leq H}\sum_{k'_1\leq 2Q}\sum_{k'_2\leq 2Q}\eta_0(h)\eta_1(k'_1)\eta_2(k'_2)\\
&\times e\left(\frac{a h \overline{\gamma} \overline{d}\overline{d_1} (d_1\nu'_1-\nu_2)(\overline{\nu_2 k'_1})}{\nu'_1 k'_2}\right)\bigg{|},
\end{align*}
where the $\eta_i$ are 1-bounded functions, the $d$'s and $\delta$'s and $\gamma$ are small parameters, and the $\alpha_i$ are congruence requirements for the $\nu$'s and $k$'s mod $\gamma dd_1$, and
$$H=4L^4Q^2M^{-1}.$$

We now translate to the notation of \cite{BC} to the current task (the former is on the left of the arrow, the latter on the right) and apply Theorem \ref{BCimprovement} with 
$$\vartheta \to a(d_1\nu'_1 - \nu_2), \qquad a \to \underline{h}, \qquad m \to \gamma d d_1 \nu_2 \underline{k'_1}, \qquad n \to  \underline{k_2}, \qquad\nu'_1 \to R,$$
where the underlined terms are the ones that are considered to be non-fixed.  This makes our translation
$$|\vartheta| \ll |a| DN+|a|N, \qquad A \to H, \qquad
M \to \gamma d d_1 \nu_2 Q\ll NQ^{1+\epsilon}, \qquad N\to Q, \qquad R\to N.$$
(The authors of \cite{FR} actually break the sums into dyadic intervals and then bound the intervals appropriately, but this step only multiplies a factor of $M^\epsilon$ to the final bound, so we will skip it.)

Then
\begin{align*}
W^{Err1}\ll &\mathcal L^{200}D^{C_5}MQ^{-2}N^2QH^\frac 12\left(1+\frac{|a|}{MN}\right)^\frac 14\\
&H^\frac 12N^\frac 34Q\left(\frac{1}{Q^{\frac 18}}+\frac{1}{(QN)^{\frac 18}}+\frac{1}{(NQ)^{\frac{1}{20}}H^{\frac{3}{10}}}+\frac{1}{H^{\frac{3}{20}}N^\frac 15Q^{\frac 1{20}}}+\frac{1}{N^\frac 18Q^\frac 18}\right)\\
\ll &\mathcal L^{200}D^{C_5}MN^{\frac{11}{4}}H\left(1+\frac{|a|}{MN}\right)^\frac 14\left(\frac{1}{Q^{\frac 18}}+\frac{1}{(NQ)^{\frac{1}{20}}H^{\frac{3}{10}}}+\frac{1}{H^{\frac{3}{20}}N^\frac 15Q^{\frac 1{20}}}\right)\\
\ll &\mathcal D^{C_5}X^\epsilon\left(Q^{\frac{15}{8}}N^{\frac{11}{4}}+M^\frac{3}{20}Q^{\frac{27}{20}}N^{\frac{27}{10}} +M^\frac{3}{20}Q^{\frac{33}{20}}N^{\frac{51}{20}}\right),
\end{align*}
where the penultimate line eliminates the dominated $\frac{1}{(QN)^{\frac 18}}$-terms and the last line substitutes the definition of $H$.

Since $N<Q$, the second term is dominated by the third one, and hence
\begin{align}\label{FRimprovement}W^{Err1}
\ll &\mathcal D^{C_5}X^\epsilon\left(Q^{\frac{15}{8}}N^{\frac{11}{4}}+M^\frac{3}{20}Q^{\frac{33}{20}}N^{\frac{51}{20}}\right).
\end{align}
In \cite{FR}, the authors find
\begin{align*}W^{Err1}\tag{\cite[(49)]{FR}}
\ll &\mathcal D^{C_5}X^\epsilon\left(Q^{\frac{15}{8}}N^{\frac{23}{8}}+Q^{\frac{33}{20}}M^\frac{3}{10}N^{\frac{59}{20}}\right)
\end{align*}
Hence, we can replace \cite[(49)]{FR} with our new bound (\ref{FRimprovement}) in their Theorem 1.1 to find
\begin{align*}&|\Delta(\alpha,\beta,M,N,Q,a)|\\
&\ll ||\alpha||\left(MQ^{-1}\mathcal E^*(\beta,N,Q)+(\log X)^\kappa N^2Q+(\log X)^\kappa N^2D^{-\frac 12}M+D^CX^\epsilon\left(Q^{\frac{15}{8}}N^{\frac{11}{4}}+M^\frac{3}{20}Q^{\frac{33}{20}}N^{\frac{51}{20}}\right)
\right)^\frac 12.
\end{align*}
Since $\beta$ is Siegel-Walfisz and $\alpha$ and $\beta$ are $k$-fold divisor-bounded, we have $\mathcal E^*(\beta,N,Q)\ll N^2Q\log^{-A}X$ for every $A>0$ and thus
\begin{align*}&|\Delta(\alpha,\beta,M,N,Q,a)|\ll X\log^{-A} X+(\log^{\eta} X)N(\sqrt{MQ}+D^{-\frac 12}M)+D^{C'}X^\epsilon M^\frac 12Q^{\frac{15}{16}}N^{\frac{11}{8}}+Q^{\frac{33}{40}}M^\frac{23}{40}N^{\frac{51}{40}}.
\end{align*}
\end{proof}

\section{Proof of Corollary \ref{cor1.1.5}: unbalanced convolutions}
Beginning with the proof of (i), take
$$D=((\log 2X)^{\epsilon}/10).$$
Hence
$$X\log^{-A} X+(\log^{\eta} X)N(\sqrt{MQ}+D^{-\frac 12}M)\ll X\log^{-A} X.$$
By Shiu's theorem \cite{Sh}, since $|\alpha_m|\leq \tau_k(m)$, 
$$||\alpha||\ll M^{\frac 12}(\log X)^{(k^2-1)}.$$
So we require 
\begin{gather}\label{MQN1}M^\frac 12Q^{\frac{15}{16}}N^{\frac{11}{8}}<X^{1-\epsilon},\end{gather}
and
\begin{gather}\label{MQN2}M^\frac{23}{40}Q^{\frac{33}{40}}N^{\frac{51}{40}}<X^{1-\epsilon}.\end{gather}
Recalling that $MN\leq X$, (\ref{MQN1}) will be true when
\begin{gather}Q^{\frac{15}{16}}N^{\frac{7}{8}}<X^{\frac 12-\epsilon},\end{gather}
while (\ref{MQN2}) is true when
\begin{gather}Q^{\frac{33}{40}}N^{\frac{28}{40}}<X^{\frac{17}{40}-\epsilon}.\end{gather}
Solving both for $N$ gives
\begin{gather*}
N\leq Q^{-\frac{15}{14}}X^{\frac 47-\epsilon},\\
N\leq Q^{-\frac{33}{28}}X^{\frac{17}{28}-\epsilon}.
\end{gather*}
The latter is clearly the more restrictive as long as $Q\geq \sqrt X$.  This completes part (i) of the theorem.

For (ii), we recall a result of Fouvry \cite[Corollaire 1]{FouvII} that states that (\ref{firstquestion}) holds when $1<|a|<X/3$ and
$$Q\leq \min(\sqrt{NX},X^\frac 47N^{-\frac 67})\mbox{ }and\mbox{ }N>X^\epsilon.$$
Taking $Q\leq X^{45/89-\epsilon}$, our theorem allows us to take $N\leq X^{1/89-\epsilon}$, while Fouvry's result allows $ X^{1/89-\epsilon}\leq N\leq X^{7/90-\epsilon}$.  Similarly, for (iii), another result of Fouvry \cite[Théorème 1]{FouvI} allows \ref{firstquestion} to hold when $1<|a|<{\epsilon/1000}$ and
$$Q\leq \min(\sqrt{NX},X^\frac 58N^{-\frac 34})\mbox{ }and\mbox{ }N>X^\epsilon.$$
Our result again then gives $N\leq X^{1/89-\epsilon}$, while Fouvry's result allows $ X^{1/89-\epsilon}\leq N\leq X^{101/630-\epsilon}$.

\end{document}